\documentclass[11pt]{article}

\usepackage{amsmath,amssymb,amsthm,amsfonts}
\usepackage{geometry}

\usepackage{hyperref}

\theoremstyle{plain}
\newtheorem{theorem}{Theorem}[section]

\newtheorem{corollary}[theorem]{Corollary}
\newtheorem{proposition}[theorem]{Proposition}

\theoremstyle{definition}
\newtheorem{definition}[theorem]{Definition}

\theoremstyle{remark}
\newtheorem{remark}[theorem]{Remark}

\numberwithin{equation}{section}

\title{Entrywise Loewner Preservers on Min and Max Matrix Cones}
\author{Wei Xie}
\date{}

\begin{document}
	
	\maketitle
	
\begin{abstract}
Let
\[
A_{\min}(x)=\bigl(x_{\min(i,j)}\bigr)_{i,j=1}^n,
\qquad
A_{\max}(x)=\bigl(x_{\max(i,j)}\bigr)_{i,j=1}^n
\]
be the Min and Max matrices generated by a real sequence
\(x=(x_1,\ldots,x_n)\). Using their classical cone parametrizations, we give
exact characterizations of entrywise maps preserving positive
semidefiniteness, total nonnegativity, Loewner order, and Loewner
convexity.

Our main results concern the Loewner structure.Without assuming continuity or differentiability, entrywise
Loewner-order preservation is equivalent to \(f\) being nondecreasing and
convex. On the Min and Max cones, requiring the
Loewner-convexity inequality on arbitrary pairs is rigid and forces \(f\) to
be affine. Under
the standard convention of restricting the inequality to
Loewner-comparable pairs, the condition automatically forces
\(f\in C^1([0,\infty))\) and is equivalent to convexity of both \(f\) and
\(f'\). The analogous statement holds for Loewner concavity. In each case,
the condition is already detected in dimension two.

We also show that \(f:[0,\infty)\to\mathbb R\) preserves positive
semidefiniteness entrywise on all positive semidefinite Min or Max matrices
if and only if \(f\) is nonnegative and nondecreasing; the same condition
characterizes total-nonnegativity preservation. Finally, we determine the
power-function ranges and Loewner-order automorphisms, characterize the
entrywise preservers of the strict Min and Max classes, and relate these
strict cones to inverse \(M\)-matrices and oscillatory matrices.
\end{abstract}

	
	
	\section{Introduction}
	
	Min and Max matrices are classical structured matrices. The standard Min
	and Max matrices are \((\min(i,j))_{i,j=1}^n\) and
	\((\max(i,j))_{i,j=1}^n\), respectively. More generally, for a sequence
	\(x=(x_1,\ldots,x_n)\), one may consider
	\[
	A_{\min}(x)=\bigl(x_{\min(i,j)}\bigr)_{i,j=1}^n,
	\qquad
	A_{\max}(x)=\bigl(x_{\max(i,j)}\bigr)_{i,j=1}^n.
	\]
	
	Throughout the paper, \(A\succeq0\) and \(A\succ0\) mean that \(A\) is
	positive semidefinite and positive definite, respectively, while
	\(A\ge_{\rm e}0\) means that \(A\) is entrywise nonnegative. A matrix is
	called totally nonnegative, abbreviated TN, if all of its minors are
	nonnegative. We use the terms \emph{totally nonnegative} and
	\emph{totally positive} for matrices whose minors are, respectively,
	nonnegative and positive.

	Min matrices can be traced back to the problem book of P\'olya and
	Szeg\H{o}~\cite{PolyaSzego}. The classical Min matrix
	\((\min(i,j))_{i,j=1}^n\) also appears in probability and statistics as the
	covariance matrix of Brownian motion at equally spaced time points.
	Moy\'e~\cite{Moye} studied this covariance matrix in a statistical setting.
	Motivated by that work, Neudecker, Trenkler and
	Liu~\cite{NeudeckerTrenklerLiu} considered the more general matrix
	\(\bigl(a_{\min(i,j)}\bigr)_{i,j=1}^n\) and posed problems concerning its
	determinant, inverse and positive definiteness. These problems were later
	answered by Chu et al.~\cite{ChuPuntanenStyan}, who gave explicit
	determinant and inverse formulas and positive-definiteness criteria.
Stuart~\cite{Stuart2015} studied matrices of the same form under the name
of nested matrices and showed that those generated by strictly increasing
positive sequences are inverse \(M\)-matrices whose inverses are symmetric
irreducible tridiagonal \(M\)-matrices.

	Mattila and Haukkanen~\cite{MattilaHaukkanen2016} studied MIN and MAX
	matrices associated with an ordered multiset \(z_1\le z_2\le\cdots\le z_n\).
	They interpreted \((\min(z_i,z_j))_{i,j=1}^n\) and
	\((\max(z_i,z_j))_{i,j=1}^n\) as special meet and join matrices, and
	obtained formulas for determinants, inverses, inertia and positive
	definiteness. K{\i}l{\i}\c{c} and Ar{\i}kan~\cite{KilicArikan2019}
	studied matrices of the form \(\bigl(a_{\min(i,j)}\bigr)_{i,j\ge1}\) and
	\(\bigl(a_{\max(i,j)}\bigr)_{i,j\ge1}\), together with their reciprocal
	analogues, and derived LU decompositions, Cholesky decompositions, inverse
	matrices and LU decompositions of the inverses.
	
	Subsequent work has treated Min and Max matrices generated by special
	sequences~\cite{SolmazBahsi2022,BahsiSolak2015,Xie2024}, as well as
	\(r\)-min, \(r\)-max, and geometric variants
	\cite{KizilatesTerzioglu2022,FonsecaKizilatesTerzioglu2024}. Spectral and
	numerical aspects have also been studied, including tridiagonal reductions,
	Sturm-type eigenvalue counts, and accurate bidiagonal factorizations
	\cite{Fonseca2007,Mersin2023,KhiarMainarRoyo2025}.

	For positive semidefinite matrices, entrywise preservation has been a
	natural and important line of research. On the full positive semidefinite
	cone, the theory is classical and quite restrictive.
	Schoenberg~\cite{Schoenberg} initiated an important line of work on
	positive definite functions and entrywise preservers. Christensen and
	Ressel~\cite{ChristensenRessel} related entrywise positivity preservers in
	all dimensions to absolute monotonicity, namely nonnegativity of all
	derivatives in the interior. FitzGerald and Horn~\cite{FitzGeraldHorn}
	proved that, for \(n\times n\) positive semidefinite matrices with
	nonnegative entries, the entrywise power map
	\[
	A\mapsto A^{\circ \alpha}
	\]
	preserves positive semidefiniteness if and only if, for \(n\ge2\),
	\[
	\alpha\in\mathbb Z_{\ge0}
	\quad\text{or}\quad
	\alpha\ge n-2.
	\]
	Here and in this literature comparison, the zeroth power uses the convention
	\(0^0=1\).
	Guillot, Khare and Rajaratnam~\cite{GuillotKhareRajaratnam} further
	characterized the critical exponents for Hadamard powers preserving Loewner
	positivity, Loewner monotonicity and Loewner convexity. For \(n\times n\)
	positive semidefinite matrices with nonnegative entries, apart from the
	exceptional nonnegative integer powers, the corresponding thresholds are
	\(n-2\), \(n-1\) and \(n\).
	
	Min matrices exhibit a different behavior under entrywise maps.
	Bhatia~\cite{Bhatia} studied their infinite divisibility and observed more
	generally that positive-valued nondecreasing functions preserve positive
	semidefiniteness entrywise on Min kernels. This suggests that the exact
	entrywise preserver problem on the Min and Max classes differs
	substantially from the corresponding problem on the full positive
	semidefinite cone. Our formulation below includes necessity, the boundary
	at zero, the equivalent total-nonnegativity statement, and the Max
	analogue.

	This paper studies Min and Max matrices generated by arbitrary sequences,
	with emphasis on several associated positivity structures. One theme is the
	action of entrywise functions on positive semidefiniteness and total
	nonnegativity, together with closure under Hadamard products. Since
	applying a function entrywise to a Min matrix again gives a Min matrix, and
	similarly for Max matrices, these preserver problems can be reduced to
	questions about the generating sequence. Positive-definiteness and
	positive-semidefiniteness criteria for Min and Max matrices are already
	available; it is therefore natural to ask whether replacing positive
	semidefiniteness by total nonnegativity leads to the same entrywise
	preservers.
	
Another theme concerns strict Min and Max matrices, namely those generated
by strictly monotone positive sequences. These matrices connect the present
setting with the classical theories of totally nonnegative matrices, inverse
\(M\)-matrices, and oscillation matrices; see, for example,
Pinkus~\cite{Pinkus}, Gantmacher--Krein~\cite{GantmacherKrein}, and
Stuart~\cite{Stuart2015}. We record the structural equivalences needed to
formulate and solve the entrywise preserver problem on the strict classes.

Finally, motivated by the classical parallel study of Loewner positivity,
monotonicity, and convexity for entrywise powers, we consider Loewner-order
preservation and Loewner convexity on comparable pairs in the Min and Max
cones. The Loewner order on these cones can be described in terms of
differences of the generating sequences, leading to conditions on the
increments and Jensen gaps of the entrywise function. We also consider the
corresponding Loewner concavity problem and Loewner-order automorphisms. To
clarify the role of comparability, we contrast this standard convention with
the stronger requirement imposed on arbitrary pairs, which turns out to be
rigid.


	\section{Min and Max positivity cones}\label{sec:cones}
	
	We begin with the basic cone terminology and the standard factorizations
	used throughout the paper.
	
	\begin{definition}
		A subset \(C\) of a real vector space is called a convex cone if
		\(\alpha x+\beta y\in C\) whenever \(x,y\in C\) and
		\(\alpha,\beta\ge0\). If
		\[
		C=\operatorname{cone}\{v_1,\ldots,v_m\}
		=
		\left\{
		\alpha_1v_1+\cdots+\alpha_mv_m:
		\alpha_i\ge0
		\right\}
		\]
		with \(v_1,\ldots,v_m\) linearly independent, then \(C\) is called a
		simplicial cone.
	\end{definition}
	
	\begin{definition}
		For \(x=(x_1,\ldots,x_n)\in\mathbb R^n\), define
		\[
		A_{\min}(x)=\bigl(x_{\min(i,j)}\bigr)_{i,j=1}^n,
		\qquad
		A_{\max}(x)=\bigl(x_{\max(i,j)}\bigr)_{i,j=1}^n.
		\]
		We call them the Min and Max matrices generated by \(x\).
	\end{definition}
	
	For \(x=(x_1,\ldots,x_n)\), put \(x_0=0\) and
	\(\Delta_k=x_k-x_{k-1}\) for \(k=1,\ldots,n\). Let \(L_n\) be the lower
	triangular matrix with all entries on and below the diagonal equal to \(1\):
	\[
	L_n=
	\begin{pmatrix}
		1&0&0&\cdots&0\\
		1&1&0&\cdots&0\\
		1&1&1&\cdots&0\\
		\vdots&\vdots&\vdots&\ddots&\vdots\\
		1&1&1&\cdots&1
	\end{pmatrix}.
	\]
	Then
	\begin{equation}\label{eq:min-factorization}
		A_{\min}(x)=L_nD_xL_n^T,
		\qquad
		D_x=\operatorname{diag}(\Delta_1,\ldots,\Delta_n).
	\end{equation}
	Indeed,
	\[
	(L_nD_xL_n^T)_{ij}
	=
	\sum_{k=1}^{\min(i,j)}\Delta_k
	=
	x_{\min(i,j)}.
	\]
	
	Equivalently, for \(k=1,\ldots,n\), let
	\(u_k=(0,\ldots,0,1,\ldots,1)^T\), where the first \(k-1\) entries are
	\(0\), and the entries from \(k\) to \(n\) are \(1\). Then
	\[
	A_{\min}(x)=\sum_{k=1}^n \Delta_k u_ku_k^T.
	\]

	\begin{theorem}\label{thm:min-psd-tn}
		For \(x\in\mathbb R^n\), the following are equivalent:
		\begin{enumerate}
			\item \(A_{\min}(x)\succeq0\).
			\item \(A_{\min}(x)\) is TN.
			\item \(0\le x_1\le x_2\le\cdots\le x_n\).
		\end{enumerate}
	\end{theorem}
	
	\begin{proof}
		If
		\[
		0\le x_1\le\cdots\le x_n,
		\]
		then \(\Delta_k\ge0\) for all \(k\). Hence, by
		\eqref{eq:min-factorization}, \(A_{\min}(x)\succeq0\).
		
		Moreover, \(L_n\) and \(D_x\) are totally nonnegative. Since products of
		totally nonnegative matrices are totally nonnegative \cite{Pinkus}, it
		follows that \(A_{\min}(x)\) is TN.
		
		Conversely, if \(A_{\min}(x)\succeq0\), then
		\[
		D_x=L_n^{-1}A_{\min}(x)L_n^{-T}
		\]
		satisfies \(D_x\succeq0\). Hence \(\Delta_k\ge0\) for all \(k\), which
		gives
		\[
		0\le x_1\le\cdots\le x_n.
		\]
		
		Finally, if \(A_{\min}(x)\) is TN, then all principal minors are
		nonnegative. Since \(A_{\min}(x)\) is symmetric, the principal-minor
		criterion implies \(A_{\min}(x)\succeq0\).
	\end{proof}
	
	Theorem~\ref{thm:min-psd-tn} shows that the positive semidefinite Min
	matrices are exactly those generated by nonnegative nondecreasing
	sequences. This class is closed under addition and multiplication by
	nonnegative scalars. We therefore make the following definition.
	
	\begin{definition}
		The positive semidefinite Min cone is
		\[
		\mathcal C_{\min,n}
		=
		\{A_{\min}(x):0\le x_1\le\cdots\le x_n\}.
		\]
	\end{definition}
	
	\begin{proposition}\label{prop:min-simplicial}
		The positive semidefinite Min cone is the simplicial cone
		\[
		\mathcal C_{\min,n}
		=
		\operatorname{cone}\{u_1u_1^T,\ldots,u_nu_n^T\}.
		\]
		Equivalently, every \(A_{\min}(x)\in\mathcal C_{\min,n}\) has the unique
		representation
		\[
		A_{\min}(x)=\sum_{k=1}^n \Delta_k u_ku_k^T,
		\qquad
		\Delta_k=x_k-x_{k-1}\ge0,
		\]
		and the generators \(u_1u_1^T,\ldots,u_nu_n^T\) are linearly independent.
	\end{proposition}
	
	\begin{proof}
		The representation follows from \eqref{eq:min-factorization}. To prove
		uniqueness, suppose
		\[
		\sum_{k=1}^n c_k u_ku_k^T=0.
		\]
		The left-hand side is a Min matrix whose successive increments are
		precisely \(c_1,\ldots,c_n\). Since the resulting matrix is zero, its
		generating sequence is identically zero, and all increments vanish.
		Therefore \(c_1=\cdots=c_n=0\). Thus the generators are linearly
		independent, and \(\mathcal C_{\min,n}\) is simplicial.
	\end{proof}
	
	The determinant is also immediate from the factorization:
	\[
	\det A_{\min}(x)
	=
	\det D_x
	=
	x_1\prod_{k=2}^n (x_k-x_{k-1}).
	\]
	
	The corresponding Max statements are obtained by reversing the order of the
	indices. Let \(J_n\) be the reversal matrix, that is,
	\[
	(J_n)_{ij}=
	\begin{cases}
		1, & i+j=n+1,\\
		0, & \text{otherwise}.
	\end{cases}
	\]
	If \(y_k=x_{n+1-k}\), then
	\begin{equation}\label{eq:max-min-reversal}
		J_nA_{\max}(x)J_n=A_{\min}(y).
	\end{equation}
	
Consequently, \(A_{\max}(x)\succeq0\), \(A_{\max}(x)\) is TN, and
\[
x_1\ge x_2\ge\cdots\ge x_n\ge0
\]
are equivalent. Total nonnegativity is preserved under simultaneous reversal
because the row and column permutations contribute the same sign to each
minor. We therefore write
\[
\mathcal C_{\max,n}
=
\{A_{\max}(x):x_1\ge x_2\ge\cdots\ge x_n\ge0\}
\]
for the positive semidefinite Max cone.
	
	For \(k=1,\ldots,n\), let
	\(v_k=(1,\ldots,1,0,\ldots,0)^T\), where the first \(k\) entries are \(1\)
	and the remaining entries are \(0\). Put \(x_{n+1}=0\) and
	\(\nabla_k=x_k-x_{k+1}\) for \(k=1,\ldots,n\). Then
	\[
	A_{\max}(x)=\sum_{k=1}^n \nabla_k v_kv_k^T.
	\]
	Indeed, the coefficient \(\nabla_k\) contributes to the \((i,j)\)-entry
	precisely when \(k\ge \max(i,j)\), and hence
	\[
	\sum_{k=\max(i,j)}^n \nabla_k
	=
	x_{\max(i,j)}.
	\]
	
Thus \(\mathcal C_{\max,n}\) is likewise the simplicial cone generated by
\(v_1v_1^T,\ldots,v_nv_n^T\); uniqueness follows either from the displayed
representation or from Proposition~\ref{prop:min-simplicial} by simultaneous
reversal.
	
	\section{Entrywise preservers and Hadamard products}\label{sec:entrywise}
	
	\begin{definition}
		Let \(f:[0,\infty)\to\mathbb R\). For a matrix \(A=(a_{ij})\) with entries
		in \([0,\infty)\), define \(f[A]=(f(a_{ij}))\). We say that \(f\)
		preserves a matrix class \(\mathcal S\) entrywise if \(A\in\mathcal S\)
		implies \(f[A]\in\mathcal S\).
	\end{definition}
	
	For matrices of the same size, \(A\circ B\) denotes their Hadamard product.
	
	\begin{proposition}\label{prop:hadamard-closure}
		The cones \(\mathcal C_{\min,n}\) and \(\mathcal C_{\max,n}\) are closed
		under Hadamard products. More precisely,
		\[
		A_{\min}(x)\circ A_{\min}(y)=A_{\min}(xy),
		\]
		and
		\[
		A_{\max}(x)\circ A_{\max}(y)=A_{\max}(xy),
		\]
		where \(xy=(x_1y_1,\ldots,x_ny_n)\).
	\end{proposition}
	
	\begin{proof}
		For Min matrices,
		\[
		(A_{\min}(x)\circ A_{\min}(y))_{ij}
		=
		x_{\min(i,j)}y_{\min(i,j)}
		=
		(xy)_{\min(i,j)}.
		\]
		If \(x\) and \(y\) are nonnegative and nondecreasing, then \(xy\) is also
		nonnegative and nondecreasing. The Max case is identical, with
		nonincreasing sequences in place of nondecreasing sequences.
	\end{proof}
	
	\begin{remark}
		The Hadamard-product closure is a special case of the corresponding
		classical result for totally nonnegative Green matrices; see
		Pinkus~\cite[Section 4.10]{Pinkus}. The positive Hadamard powers are also
		consistent with the infinite divisibility of Min matrices discussed by
		Bhatia~\cite{Bhatia}.
	\end{remark}
	
	For positive-valued functions on \((0,\infty)\), the sufficiency of
	monotonicity for preserving positive semidefiniteness on Min kernels was
	observed by Bhatia~\cite{Bhatia}. The following theorem gives the exact
	version on \([0,\infty)\), allows zero values, and includes the equivalent
	total-nonnegativity formulation.
	
	\begin{theorem}\label{thm:entrywise-preserver}
		Let \(f:[0,\infty)\to\mathbb R\). The following are equivalent:
		\begin{enumerate}
			\item For every \(n\ge1\) and every \(A\in\mathcal C_{\min,n}\), one has
			\(f[A]\succeq0\); equivalently, the same holds on every Max cone.
			\item For every \(n\ge1\) and every \(A\in\mathcal C_{\min,n}\), one has
			\(f[A]\) totally nonnegative; equivalently, the same holds on every Max
			cone.
			\item For every \(n\ge1\) and every \(A\in\mathcal C_{\min,n}\), one has
			\(f[A]\in\mathcal C_{\min,n}\); equivalently, \(f[A]\in
			\mathcal C_{\max,n}\) for every \(A\in\mathcal C_{\max,n}\).
			\item \(f(t)\ge0\) for all \(t\ge0\), and \(f\) is nondecreasing on
			\([0,\infty)\).
		\end{enumerate}
	\end{theorem}
	
	\begin{proof}
		If \(A=A_{\min}(x)\), then
		\[
		f[A]=A_{\min}(f(x_1),\ldots,f(x_n)).
		\]
		By Theorem~\ref{thm:min-psd-tn}, this matrix satisfies \(f[A]\succeq0\),
		is totally nonnegative, or belongs to \(\mathcal C_{\min,n}\), if and only if
		\[
		0\le f(x_1)\le f(x_2)\le\cdots\le f(x_n).
		\]
		Thus if \(f\ge0\) and \(f\) is nondecreasing, then
		\(f[A]\in\mathcal C_{\min,n}\) for every \(A\in\mathcal C_{\min,n}\). This
		proves the sufficiency and the equivalence of the first three conditions
		under condition \(4\).
		
		Conversely, assume condition \(1\). Let \(0\le s\le t\). Then
		\[
		A_{\min}(s,t)=
		\begin{pmatrix}
			s&s\\
			s&t
		\end{pmatrix}
		\in\mathcal C_{\min,2}.
		\]
		Hence
		\[
		f[A_{\min}(s,t)]
		=
		A_{\min}(f(s),f(t))
		\]
		satisfies \(A_{\min}(f(s),f(t))\succeq0\). By
		Theorem~\ref{thm:min-psd-tn}, applied in dimension \(2\), this implies
		\[
		0\le f(s)\le f(t).
		\]
		Since \(0\le s\le t\) was arbitrary, \(f\ge0\) and \(f\) is nondecreasing
		on \([0,\infty)\). Thus condition \(4\) holds. The Max equivalences follow
		by simultaneous reversal of rows and columns.
	\end{proof}
	
	\begin{remark}
		In Theorem~\ref{thm:entrywise-preserver}, the assumption over all dimensions may
		be replaced by the assumption in any fixed dimension \(n\ge2\). The case
		\(n=2\) already yields the required nonnegativity and monotonicity.
	\end{remark}
	
	\begin{remark}
		Theorem~\ref{thm:entrywise-preserver} allows discontinuous preservers.
		For instance, nonnegative nondecreasing step functions preserve positive
		semidefiniteness and total nonnegativity on all Min and Max cones. This is
		in contrast with the classical full positive semidefinite cone, where
		dimension-free entrywise preservers are subject to absolute
		monotonicity-type restrictions; see, for example,
		\cite{Schoenberg,ChristensenRessel}.
	\end{remark}
	
	\section{Strict cones, inverse \texorpdfstring{\(M\)}{M}-matrices and oscillatory matrices}
	\label{sec:strict}

	\subsection{Structural equivalences}
	
	\begin{definition}
		We call the relative interiors of the Min and Max cones the strict Min and
		Max cones. Explicitly,
		\[
		\mathcal C_{\min,n}^{\circ}
		=
		\{A_{\min}(x):0<x_1<x_2<\cdots<x_n\},
		\]
		and
		\[
		\mathcal C_{\max,n}^{\circ}
		=
		\{A_{\max}(x):x_1>x_2>\cdots>x_n>0\}.
		\]
		These are relative interiors in their natural linear spans.
	\end{definition}
	
	\begin{definition}
		A \(Z\)-matrix is a real matrix whose off-diagonal entries are nonpositive.
		A nonsingular \(M\)-matrix is a nonsingular \(Z\)-matrix whose inverse is
		entrywise nonnegative. Equivalently, a nonsingular \(M\)-matrix can be
		written as
		\[
		sI-B,
		\qquad
		B\ge_{\rm e}0,
		\qquad
		s>\rho(B).
		\]
		A matrix \(A\) is called a nonsingular inverse \(M\)-matrix if \(A^{-1}\)
		is a nonsingular \(M\)-matrix.
	\end{definition}
	
	We shall use the standard fact that a symmetric \(Z\)-matrix is a
	nonsingular \(M\)-matrix if and only if it is positive definite; see
	\cite{BermanPlemmons}.
	
	\begin{definition}
		A matrix is totally positive, abbreviated TP, if all minors of all orders
		are positive. A totally nonnegative matrix \(A\) is called oscillatory if
		there exists a positive integer \(m\) such that \(A^m\) is totally positive.
	\end{definition}
	
	We shall use the classical criterion that a nonsingular totally
	nonnegative matrix whose first superdiagonal and first subdiagonal entries
	are positive is oscillatory; see
	Gantmacher--Krein~\cite{GantmacherKrein} and
	Pinkus~\cite[Theorem 5.2]{Pinkus}. The same theorem implies that the
	\((n-1)\)-st power of every \(n\times n\) oscillatory matrix is totally
	positive.
	
	If
	\[
	0<x_1<x_2<\cdots<x_n,
	\]
	then all increments \(\Delta_1=x_1\) and
	\(\Delta_k=x_k-x_{k-1}\) for \(k=2,\ldots,n\) are positive. Hence
	\(A_{\min}(x)\) is nonsingular, and
	\[
	A_{\min}(x)^{-1}
	=
	L_n^{-T}D_x^{-1}L_n^{-1}.
	\]
	Writing \(q_k=1/\Delta_k>0\) for \(k=1,\ldots,n\), one obtains, for
	\(n\ge2\),
	\begin{equation}\label{eq:min-inverse}
		A_{\min}(x)^{-1}
		=
		\begin{pmatrix}
			q_1+q_2 & -q_2 & 0 & \cdots & 0\\
			-q_2 & q_2+q_3 & -q_3 & \cdots & 0\\
			0 & -q_3 & q_3+q_4 & \cdots & 0\\
			\vdots & \vdots & \vdots & \ddots & -q_n\\
			0 & 0 & 0 & -q_n & q_n
		\end{pmatrix}.
	\end{equation}
	For \(n=1\), the inverse is simply \([q_1]\). For \(n\ge2\), when
	\(0<x_1<\cdots<x_n\), all \(q_k\) are positive, and hence
	\(A_{\min}(x)^{-1}\) is a symmetric irreducible tridiagonal \(Z\)-matrix.
	In all cases, since \(A_{\min}(x)\succ0\), also
	\(A_{\min}(x)^{-1}\succ0\). Thus \(A_{\min}(x)^{-1}\) is a nonsingular
	\(M\)-matrix; for \(n\ge2\), it is a symmetric irreducible tridiagonal
	nonsingular \(M\)-matrix.
	
	The next theorem collects, in the present notation, the classical
	positivity consequences of the Min factorization and the oscillatory
	criterion. It is included to identify the relative interior on which the
	strict entrywise preserver problem is considered.
	
	\begin{theorem}\label{thm:strict-equivalences}
		For a Min matrix
		\[
		A_{\min}(x)=\bigl(x_{\min(i,j)}\bigr)_{i,j=1}^n,
		\]
		the following are equivalent:
		\begin{enumerate}
			\item \(0<x_1<x_2<\cdots<x_n\).
			\item \(A_{\min}(x)\succ0\).
			\item \(A_{\min}(x)\) is nonsingular and totally nonnegative.
			\item \(A_{\min}(x)\) is a nonsingular inverse \(M\)-matrix.
			\item \(A_{\min}(x)\) is oscillatory.
		\end{enumerate}
		The corresponding five conditions for a Max matrix are equivalent after
		replacing the first condition by
		\(x_1>x_2>\cdots>x_n>0\) and replacing \(A_{\min}\) by \(A_{\max}\).
	\end{theorem}
	
	\begin{proof}
		The equivalence between \(1\) and \(2\) follows from
		\eqref{eq:min-factorization}. Since \(L_n\) is nonsingular,
		\(A_{\min}(x)\succ0\) if and only if \(D_x\succ0\). This is equivalent to
		\[
		\Delta_k>0,
		\qquad k=1,\ldots,n,
		\]
		that is,
		\[
		0<x_1<x_2<\cdots<x_n.
		\]
		
		By Theorem~\ref{thm:min-psd-tn}, total nonnegativity of
		\(A_{\min}(x)\) is equivalent to
		\[
		0\le x_1\le x_2\le\cdots\le x_n.
		\]
		Moreover,
		\[
		\det A_{\min}(x)
		=
		x_1\prod_{k=2}^n (x_k-x_{k-1}).
		\]
		Therefore \(A_{\min}(x)\) is nonsingular and totally nonnegative if and
		only if
		\[
		0<x_1<x_2<\cdots<x_n.
		\]
		Thus \(1\) and \(3\) are equivalent.
		
		Assume now that \(1\) holds. By \eqref{eq:min-inverse},
		\(A_{\min}(x)^{-1}\) is a symmetric \(Z\)-matrix for \(n\ge2\), while the
		case \(n=1\) is trivial. Since \(A_{\min}(x)\succ0\), also
		\(A_{\min}(x)^{-1}\succ0\). Hence \(A_{\min}(x)^{-1}\) is a nonsingular
		\(M\)-matrix, and \(A_{\min}(x)\) is a nonsingular inverse \(M\)-matrix.
		Thus \(1\) implies condition \(4\).
		
		Conversely, assume condition \(4\). Then \(A_{\min}(x)^{-1}\) is a
		nonsingular \(M\)-matrix. Since \(A_{\min}(x)\) is symmetric, its inverse is
		also symmetric. Hence \(A_{\min}(x)^{-1}\succ0\), and therefore
		\(A_{\min}(x)\succ0\). Thus condition \(2\) holds, and consequently
		condition \(1\) follows.
		
		Finally, suppose \(1\) holds. Then \(A_{\min}(x)\) is nonsingular and
		totally nonnegative. Moreover,
		\[
		(A_{\min}(x))_{i,i+1}=x_i>0,
		\qquad
		(A_{\min}(x))_{i+1,i}=x_i>0,
		\qquad i=1,\ldots,n-1.
		\]
		By the oscillatory matrix criterion recalled above, \(A_{\min}(x)\) is
		oscillatory. Conversely, if \(A_{\min}(x)\) is oscillatory, then it is
		nonsingular and totally nonnegative. Hence condition \(3\) holds.

		For the Max case, put \(y_k=x_{n+1-k}\). Then
		\(J_nA_{\max}(x)J_n=A_{\min}(y)\). Positive definiteness,
		nonsingularity, total nonnegativity, the inverse \(M\)-matrix property,
		and oscillation are all preserved under simultaneous reversal. The Max
		equivalences therefore follow from the Min case.
	\end{proof}
	
	\begin{remark}
		In the Min and Max classes, positive definiteness is equivalent to
		nonsingular total nonnegativity, but the latter cannot be strengthened to
		total positivity when \(n\ge3\).
		
		Indeed, for a Min matrix, the minor determined by rows \(1,2\) and columns
		\(2,3\) is
		\[
		\det
		\begin{pmatrix}
			x_1&x_1\\
			x_2&x_2
		\end{pmatrix}
		=0.
		\]
		Thus no Min matrix of order \(n\ge3\) is totally positive. Similarly, for a
		Max matrix, the minor determined by rows \(2,3\) and columns \(1,2\) is
		\[
		\det
		\begin{pmatrix}
			x_2&x_2\\
			x_3&x_3
		\end{pmatrix}
		=0.
		\]
		Thus no Max matrix of order \(n\ge3\) is totally positive.
	\end{remark}
	
	\subsection{Entrywise preservers of the strict classes}
	
	\begin{proposition}\label{prop:strict-preserver}
		Let \(f:(0,\infty)\to\mathbb R\). The following are equivalent:
		\begin{enumerate}
			\item For every \(n\ge1\) and every
			\(A\in\mathcal C_{\min,n}^{\circ}\), one has
			\(f[A]\in\mathcal C_{\min,n}^{\circ}\).
			\item For every \(n\ge1\) and every
			\(A\in\mathcal C_{\min,n}^{\circ}\), the matrix \(f[A]\) is a
			nonsingular inverse \(M\)-matrix.
			\item \(f(t)>0\) for all \(t>0\), and \(f\) is strictly increasing on
			\((0,\infty)\).
		\end{enumerate}
		The same equivalence holds with \(\mathcal C_{\min,n}^{\circ}\) replaced by
		\(\mathcal C_{\max,n}^{\circ}\).
	\end{proposition}
	
	\begin{proof}
		Assume first that \(f(t)>0\) for all \(t>0\) and that \(f\) is strictly
		increasing. If \(A=A_{\min}(x)\in\mathcal C_{\min,n}^{\circ}\), then
		\[
		0<x_1<x_2<\cdots<x_n.
		\]
		Therefore
		\[
		0<f(x_1)<f(x_2)<\cdots<f(x_n),
		\]
		and hence
		\[
		f[A]=A_{\min}(f(x_1),\ldots,f(x_n))\in\mathcal C_{\min,n}^{\circ}.
		\]
		By Theorem~\ref{thm:strict-equivalences}, \(f[A]\) is a nonsingular
		inverse \(M\)-matrix.
		
		Conversely, suppose condition \(1\) holds. Let \(0<s<t\), and consider
		\[
		A=
		\begin{pmatrix}
			s&s\\
			s&t
		\end{pmatrix}
		=
		A_{\min}(s,t).
		\]
		Then \(A\in\mathcal C_{\min,2}^{\circ}\). Hence
		\[
		f[A]
		=
		\begin{pmatrix}
			f(s)&f(s)\\
			f(s)&f(t)
		\end{pmatrix}
		\in\mathcal C_{\min,2}^{\circ}.
		\]
		Thus
		\[
		0<f(s)<f(t)
		\]
		whenever \(0<s<t\). Hence \(f(t)>0\) for all \(t>0\), and \(f\) is
		strictly increasing. Since \(f[A]\) is again a Min matrix, condition \(2\),
		together with Theorem~\ref{thm:strict-equivalences}, implies condition
		\(1\). Thus the Min case is proved.
		
		The Max case follows by applying the same argument after simultaneous
		reversal of rows and columns, or equivalently by using
		Theorem~\ref{thm:strict-equivalences}.
	\end{proof}
	
\section{Loewner order preservers, convexity and automorphisms}
\label{sec:loewner}

	Here the Loewner order is the positive semidefinite order: for symmetric
	matrices \(A\) and \(B\), \(A\preceq B\) means \(B-A\succeq0\).
	
	For Min matrices, Loewner order has a simple description. If
	\(A_{\min}(x),A_{\min}(y)\in\mathcal C_{\min,n}\), then
	\[
	A_{\min}(x)\preceq A_{\min}(y)
	\]
	if and only if
	\[
	A_{\min}(y-x)\succeq0.
	\]
	By Theorem~\ref{thm:min-psd-tn}, this is equivalent to
	\[
	0\le y_1-x_1\le y_2-x_2\le\cdots\le y_n-x_n.
	\]
	
	\begin{definition}\label{def:increasing-increments}
		Let \(f:[0,\infty)\to\mathbb R\), and set
		\begin{equation}\label{eq:increment-function}
			G_f(s,d)=f(s+d)-f(s),
			\qquad s,d\ge0.
		\end{equation}
		We say that \(f\) has increasing increments if \(G_f\) is nondecreasing in
		each variable separately; that is,
		\[
		G_f(s,d)\le G_f(t,d)
		\quad\text{whenever }0\le s\le t,\ d\ge0,
		\]
		and
		\[
		G_f(s,d)\le G_f(s,e)
		\quad\text{whenever }s\ge0,\ 0\le d\le e.
		\]
		Equivalently,
		\[
		f(s+d)-f(s)
		\le
		f(t+e)-f(t)
		\]
		whenever \(0\le s\le t\) and \(0\le d\le e\).
	\end{definition}

	\begin{theorem}\label{thm:loewner-preserver}
		Let \(f:[0,\infty)\to\mathbb R\). The following are equivalent:
		\begin{enumerate}
			\item For every \(n\ge1\) and all \(A,B\in\mathcal C_{\min,n}\),
			\[
			A\preceq B
			\quad\Longrightarrow\quad
			f[A]\preceq f[B].
			\]
			Equivalently, the same implication holds on every Max cone.
			\item \(f\) has increasing increments on \([0,\infty)\).
			\item \(f\) is nondecreasing and convex on \([0,\infty)\).
		\end{enumerate}
	\end{theorem}
	
	\begin{proof}
		Assume first that \(f\) has increasing increments. Let
		\(A=A_{\min}(x)\) and \(B=A_{\min}(y)\) be elements of
		\(\mathcal C_{\min,n}\) with \(A\preceq B\). Put \(d_k=y_k-x_k\). Then
		\[
		0\le d_1\le d_2\le\cdots\le d_n.
		\]
		Moreover, \(f(y_k)-f(x_k)=G_f(x_k,d_k)\). Since \(G_f(x_k,0)=0\) and
		\(G_f\) is nondecreasing in the second variable, \(G_f(x_k,d_k)\ge0\).
		Since both \(x_k\) and \(d_k\) are nondecreasing in \(k\), and \(G_f\) is
		nondecreasing in each variable, the sequence
		\[
		f(y_k)-f(x_k)
		\]
		is nondecreasing. Hence
		\[
		0\le f(y_1)-f(x_1)\le\cdots\le f(y_n)-f(x_n).
		\]
		Therefore
		\[
		f[B]-f[A]
		=
		A_{\min}(f(y)-f(x))
		\succeq0.
		\]
		Thus \(f[A]\preceq f[B]\).
		
		Conversely, assume \(f\) preserves Loewner order on all Min cones. Let
		\[
		0\le s\le t,
		\qquad
		0\le d\le e.
		\]
		Consider
		\[
		A=A_{\min}(s,t),
		\qquad
		B=A_{\min}(s+d,t+e).
		\]
		Then
		\[
		B-A=A_{\min}(d,e)\succeq0,
		\]
		and hence \(A\preceq B\). Therefore
		\[
		f[B]-f[A]
		=
		A_{\min}(f(s+d)-f(s),\,f(t+e)-f(t))
		\succeq0.
		\]
		By Theorem~\ref{thm:min-psd-tn},
		\[
		0\le f(s+d)-f(s)\le f(t+e)-f(t).
		\]
		This is precisely the increasing-increments condition. The Max equivalence
		follows by reversing the order of indices.

		It remains to identify increasing increments with ordinary monotonicity and
		convexity. Assume first that \(f\) has increasing increments. Then, for every
		\(s,h\ge0\),
		\[
		f(s+h)-f(s)
		\le
		f(s+2h)-f(s+h).
		\]
		Equivalently,
		\[
		f(s+h)
		\le
		\frac{f(s)+f(s+2h)}2.
		\]
		Thus \(f\) is midpoint convex on \([0,\infty)\). Moreover, increasing
		increments imply monotonicity, because by \eqref{eq:increment-function},
		\[
		f(s+d)-f(s)=G_f(s,d)\ge G_f(s,0)=0
		\]
		for all \(s,d\ge0\). Hence \(f\) is nondecreasing and therefore locally
		bounded on \((0,\infty)\). A locally bounded midpoint-convex function is
		continuous and convex on the interior of its domain. Hence \(f\) is
		continuous and convex on \((0,\infty)\).

		It remains to treat the endpoint. By monotonicity, the finite right limit
		\[
		L=\lim_{s\downarrow0}f(s)
		\]
		exists and satisfies \(f(0)\le L\). For any fixed \(d>0\), increasing
		increments give
		\[
		f(d)-f(0)=G_f(0,d)\le G_f(s,d)=f(s+d)-f(s),
		\qquad s>0.
		\]
		Letting \(s\downarrow0\) and using continuity at \(d\) yields
		\[
		f(d)-f(0)\le f(d)-L,
		\]
		so \(L\le f(0)\). Thus \(L=f(0)\), and \(f\) is continuous at zero.
		Consequently, midpoint convexity implies convexity on all of
		\([0,\infty)\).
		
		Conversely, assume that \(f\) is nondecreasing and convex. For each fixed
		\(d\ge0\), convexity implies that
		\[
		s\mapsto f(s+d)-f(s)
		\]
		is nondecreasing on \([0,\infty)\). For each fixed \(s\ge0\), the map
		\[
		d\mapsto f(s+d)-f(s)
		\]
		is nondecreasing because \(f\) is nondecreasing. Therefore \(G_f(s,d)\) is
		nondecreasing in each variable separately, and \(f\) has increasing
		increments.
	\end{proof}

We next consider convexity of the induced entrywise map with respect to the
Loewner order. Requiring the convexity inequality for arbitrary pairs in the
cone is strictly stronger than the standard entrywise Loewner-convexity
convention, which imposes it on Loewner-comparable pairs; see
\cite{Hiai2009,GuillotKhareRajaratnam}. Hiai
\cite[Proposition~1.2(2)]{Hiai2009} showed that arbitrary-pair convexity on
the full positive semidefinite class forces an affine function. We first show
that the same rigidity is already present on the much smaller Min and Max
cones.

\begin{definition}\label{def:all-pairs-loewner-convexity}
	Let \(f:[0,\infty)\to\mathbb R\). The entrywise map induced by \(f\) is
	called Loewner convex on arbitrary pairs in \(\mathcal C_{\min,n}\) if
	\[
	f[\lambda A+(1-\lambda)B]
	\preceq
	\lambda f[A]+(1-\lambda)f[B]
	\]
	for all \(A,B\in\mathcal C_{\min,n}\) and \(0\le\lambda\le1\). It is
	called Loewner concave on arbitrary pairs if the reverse inequality holds. The
	corresponding notions on \(\mathcal C_{\max,n}\) are defined analogously.
\end{definition}

\begin{proposition}\label{prop:all-pairs-rigidity}
	Let \(f:[0,\infty)\to\mathbb R\). The following are equivalent:
	\begin{enumerate}
		\item The entrywise map induced by \(f\) is Loewner convex on arbitrary
		pairs in every Min cone.
		\item The entrywise map induced by \(f\) is Loewner concave on arbitrary
		pairs in every Min cone.
		\item The function \(f\) is affine on \([0,\infty)\).
	\end{enumerate}
	The same equivalence holds for the Max cones. In the first two assertions,
	dimension \(2\) is sufficient.
\end{proposition}

\begin{proof}
	Every affine function satisfies both inequalities with equality. Conversely,
	let \(0\le s<t\), and set
	\[
	A=A_{\min}(s,t),
	\qquad
	B=A_{\min}(t,t).
	\]
	For \(0\le\lambda\le1\), put
	\[
	g=\lambda f(s)+(1-\lambda)f(t)
	-f(\lambda s+(1-\lambda)t).
	\]
	A direct calculation gives
	\[
	\lambda f[A]+(1-\lambda)f[B]
	-f[\lambda A+(1-\lambda)B]
	=A_{\min}(g,0).
	\]
	Loewner convexity on arbitrary pairs implies \(A_{\min}(g,0)\succeq0\),
	while Loewner concavity on arbitrary pairs implies
	\(A_{\min}(-g,0)\succeq0\). By
	Theorem~\ref{thm:min-psd-tn}, either condition forces \(g=0\). Hence
	\[
	f(\lambda s+(1-\lambda)t)
	=\lambda f(s)+(1-\lambda)f(t)
	\]
	for all \(0\le s<t\) and \(0\le\lambda\le1\), so \(f\) is affine. The Max
	case follows by simultaneous reversal of rows and columns.
\end{proof}

We now adopt the standard comparable-pairs convention.

\begin{definition}\label{def:loewner-convexity}
	Let \(f:[0,\infty)\to\mathbb R\). The entrywise map induced by \(f\) is
	called Loewner convex on \(\mathcal C_{\min,n}\) if
	\[
	f[\lambda A+(1-\lambda)B]
	\preceq
	\lambda f[A]+(1-\lambda)f[B]
	\]
	whenever
	\[
	A,B\in\mathcal C_{\min,n},
	\qquad
	B\preceq A,
	\qquad
	0\le\lambda\le1.
	\]
	It is called Loewner concave if the reverse inequality holds. The
	corresponding notions on \(\mathcal C_{\max,n}\) are defined analogously.
\end{definition}

For \(0\le\lambda\le1\), define the Jensen gap
\[
J_{f,\lambda}(s,d)
=
\lambda f(s+d)+(1-\lambda)f(s)-f(s+\lambda d),
\qquad s,d\ge0.
\]

\begin{proposition}\label{prop:jensen-gap-characterization}
	Let \(f:[0,\infty)\to\mathbb R\). Then the following assertions hold.
	\begin{enumerate}
		\item The entrywise map induced by \(f\) is Loewner convex on every
		Min cone if and only if, for every \(0\le\lambda\le1\), the function
		\[
		(s,d)\longmapsto J_{f,\lambda}(s,d)
		\]
		is nondecreasing in each variable separately.
		
		\item The entrywise map induced by \(f\) is Loewner concave on every
		Min cone if and only if, for every \(0\le\lambda\le1\), the function
		\[
		(s,d)\longmapsto J_{f,\lambda}(s,d)
		\]
		is nonincreasing in each variable separately.
	\end{enumerate}
	The same characterizations hold for the Max cones. In each assertion,
	dimension \(2\) is sufficient for the necessity.
\end{proposition}

\begin{proof}
	Let
	\[
	A=A_{\min}(x),
	\qquad
	B=A_{\min}(y),
	\]
	where \(A,B\in\mathcal C_{\min,n}\) and \(B\preceq A\). Put
	\[
	d_k=x_k-y_k,
	\qquad k=1,\ldots,n.
	\]
	By the Loewner-order characterization for Min matrices,
	\[
	0\le d_1\le d_2\le\cdots\le d_n,
	\]
	while
	\[
	0\le y_1\le y_2\le\cdots\le y_n.
	\]
	A direct calculation gives
	\[
	\lambda f[A]+(1-\lambda)f[B]
	-
	f[\lambda A+(1-\lambda)B]
	=
	A_{\min}\bigl(
	J_{f,\lambda}(y_1,d_1),\ldots,
	J_{f,\lambda}(y_n,d_n)
	\bigr).
	\]
	
	Suppose first that \(J_{f,\lambda}\) is nondecreasing in each variable.
	Since \(J_{f,\lambda}(s,0)=0\), it is nonnegative. Since both \(y_k\) and
	\(d_k\) are nondecreasing in \(k\),
	\[
	0\le
	J_{f,\lambda}(y_1,d_1)
	\le\cdots\le
	J_{f,\lambda}(y_n,d_n).
	\]
	Theorem~\ref{thm:min-psd-tn} therefore implies
	\[
	f[\lambda A+(1-\lambda)B]
	\preceq
	\lambda f[A]+(1-\lambda)f[B].
	\]
	Thus the induced map is Loewner convex.
	
	Conversely, assume Loewner convexity in dimension \(2\). Let
	\[
	0\le s\le t,
	\qquad
	0\le d\le e,
	\]
	and set
	\[
	B=A_{\min}(s,t),
	\qquad
	A=A_{\min}(s+d,t+e).
	\]
	Then \(A,B\in\mathcal C_{\min,2}\), and
	\[
	A-B=A_{\min}(d,e)\succeq0.
	\]
	The Loewner convexity inequality gives
	\[
	A_{\min}\bigl(
	J_{f,\lambda}(s,d),
	J_{f,\lambda}(t,e)
	\bigr)
	\succeq0.
	\]
	By Theorem~\ref{thm:min-psd-tn},
	\[
	0\le
	J_{f,\lambda}(s,d)
	\le
	J_{f,\lambda}(t,e).
	\]
	Thus \(J_{f,\lambda}\) is nondecreasing in each variable.
	
	For Loewner concavity,
	\[
	f[\lambda A+(1-\lambda)B]
	-
	\lambda f[A]-(1-\lambda)f[B]
	=
	A_{\min}\bigl(
	-J_{f,\lambda}(y_1,d_1),\ldots,
	-J_{f,\lambda}(y_n,d_n)
	\bigr).
	\]
	If \(J_{f,\lambda}\) is nonincreasing in each variable, then
	\(J_{f,\lambda}(s,d)\le J_{f,\lambda}(s,0)=0\). Hence
	Theorem~\ref{thm:min-psd-tn} shows that the concavity inequality holds.
	Conversely, applying the concavity inequality in dimension \(2\)
	to
	\[
	B=A_{\min}(s,t),
	\qquad
	A=A_{\min}(s+d,t+e),
	\]
	where \(0\le s\le t\) and \(0\le d\le e\), gives
	\[
	J_{f,\lambda}(s,d)
	\ge
	J_{f,\lambda}(t,e),
	\qquad
	J_{f,\lambda}(s,d)\le0.
	\]
	Thus \(J_{f,\lambda}\) is nonincreasing in each variable.

	The Max statements follow by simultaneous reversal of rows and columns.
\end{proof}

For the full positive semidefinite cone on an open symmetric interval,
Hiai~\cite[Theorem~3.2(1)]{Hiai2009} proved that entrywise Loewner
convexity is equivalent to differentiability of \(f\) together with
Loewner monotonicity of \(f'\); see also
\cite[Theorem~4.3(1)]{GuillotKhareRajaratnam} and
\cite[Chapter~19]{Khare2022}. The next result is the
restricted-cone analogue.

\begin{theorem}\label{thm:loewner-convexity-regularity}
	Let \(f:[0,\infty)\to\mathbb R\). Under the comparable-pairs convention,
	the following assertions hold:
	\begin{enumerate}
		\item The entrywise map induced by \(f\) is Loewner convex on all Min
		cones, equivalently on all Max cones, if and only if
		\(f\in C^1([0,\infty))\) and both \(f\) and \(f'\) are convex on
		\([0,\infty)\).
		
		\item The entrywise map induced by \(f\) is Loewner concave on all Min
		cones, equivalently on all Max cones, if and only if
		\(f\in C^1([0,\infty))\) and both \(f\) and \(f'\) are concave on
		\([0,\infty)\).
	\end{enumerate}
	In either assertion, it is enough to assume the corresponding property in
	dimension \(2\).
\end{theorem}

\begin{proof}
	Suppose first that the induced map is Loewner convex. By
	Proposition~\ref{prop:jensen-gap-characterization}, the Jensen gap
	\[
	J_{f,\lambda}(s,d)
	=
	\lambda f(s+d)+(1-\lambda)f(s)-f(s+\lambda d)
	\]
	is nonnegative and nondecreasing in \(s\), for every
	\(0\le\lambda\le1\). Its nonnegativity is precisely Jensen convexity, so
	\(f\) is convex and hence continuous.

	We show that the gap monotonicity supplies the missing differentiability.
	Fix a compact interval \(K=[a,b]\subset(0,\infty)\), and choose
	\(\eta>0\) so that \([a-\eta,b+\eta]\subset(0,\infty)\). For
	\(0<\varepsilon<\eta\), let \(\rho_\varepsilon\ge0\) be an even smooth
	mollifier supported in \((-\varepsilon,\varepsilon)\), normalized by
	\(\int_{\mathbb R}\rho_\varepsilon(u)\,du=1\), and define
	\[
	f_\varepsilon(x)
	=
	\int_{\mathbb R}f(x-u)\rho_\varepsilon(u)\,du,
	\qquad x\in K.
	\]
	This is well defined because \(x-u\in[a-\eta,b+\eta]\subset(0,\infty)\)
	on the support of \(\rho_\varepsilon\). Whenever
	\(s,s+d\in K\), with the endpoints taken in the interior when a derivative
	is used,
	\[
	J_{f_\varepsilon,\lambda}(s,d)
	=
	\int_{\mathbb R}\rho_\varepsilon(u)
	J_{f,\lambda}(s-u,d)\,du.
	\]
	Indeed, every translated pair \(s-u,s+d-u\) remains in
	\([a-\eta,b+\eta]\), where the original gap condition applies. Thus this
	mollified gap is nonnegative and nondecreasing in \(s\). Since it is smooth,
	\[
	\frac{\partial}{\partial s}J_{f_\varepsilon,\lambda}(s,d)
	=J_{f_\varepsilon',\lambda}(s,d)\ge0.
	\]
	Consequently, \(f_\varepsilon'\) is convex on the interior of \(K\).

	Let \(f'_-(x)\) and \(f'_+(x)\) denote the one-sided derivatives of the
	convex function \(f\). Its almost-everywhere derivative is locally bounded
	and has these values as its left and right essential limits. Hence, by
	splitting the convolution integral at \(u=0\), evenness of the mollifier
	gives
	\[
	\lim_{\varepsilon\downarrow0}f_\varepsilon'(x)
	=
	\frac{f'_-(x)+f'_+(x)}2
	=:m(x).
	\]
	Since \(K\) was arbitrary, pointwise limits preserve the convexity
	inequality and show that \(m\) is convex on \((0,\infty)\); it is therefore
	continuous there. On the other hand, the
	standard one-sided derivative properties of convex functions give
	\[
	\lim_{y\uparrow x}m(y)=f'_-(x),
	\qquad
	\lim_{y\downarrow x}m(y)=f'_+(x).
	\]
	Continuity of \(m\) forces \(f'_-(x)=f'_+(x)\) for every \(x>0\). Hence
	\(f\) is differentiable on \((0,\infty)\), and \(f'=m\) is convex.

	The derivative also extends continuously to zero. Put \(g=f'\) on
	\((0,\infty)\). Since \(f\) is convex, \(g\) is nondecreasing, so
	\(L=\lim_{x\downarrow0}g(x)\) exists in \([-\infty,\infty)\). Convexity of
	\(g\) gives, for fixed \(0<a<b\) and \(0<x<a\),
	\[
	\frac{g(a)-g(x)}{a-x}
	\le
	\frac{g(b)-g(a)}{b-a}.
	\]
	Thus \(g\) is bounded below near zero and \(L\) is finite. The standard
	endpoint derivative formula for a finite convex function now yields
	\(f'_+(0)=L\). Defining \(f'(0)=L\) makes \(f'\) continuous and convex on
	\([0,\infty)\), and hence \(f\in C^1([0,\infty))\).

	Conversely, suppose \(f\in C^1([0,\infty))\) and both \(f\) and \(f'\)
	are convex. Then \(J_{f,\lambda}\ge0\). Moreover,
	\[
	\frac{\partial}{\partial d}J_{f,\lambda}(s,d)
	=
	\lambda\bigl(f'(s+d)-f'(s+\lambda d)\bigr)\ge0
	\]
	because \(f'\) is nondecreasing, and
	\[
	\frac{\partial}{\partial s}J_{f,\lambda}(s,d)
	=J_{f',\lambda}(s,d)\ge0
	\]
	because \(f'\) is convex. Proposition~\ref{prop:jensen-gap-characterization} now gives
	Loewner convexity. This proves the first assertion. Applying it to \(-f\)
	gives the concavity assertion.
\end{proof}

\begin{corollary}\label{cor:power-functions}
	Let \(p_\alpha(t)=t^\alpha\) for \(t\ge0\), where \(\alpha>0\). Then the
	following hold.
	\begin{enumerate}
		\item \(p_\alpha\) preserves positive semidefiniteness and total
		nonnegativity entrywise on all Min and Max cones for every
		\(\alpha>0\).
		
		\item \(p_\alpha\) preserves the strict Min and Max classes,
		equivalently the nonsingular inverse \(M\)-matrix subclasses within
		the Min and Max classes, for every \(\alpha>0\).
		
		\item \(p_\alpha\) preserves Loewner order entrywise on all Min and Max
		cones if and only if
		$
		\alpha\ge1.
		$
		
		\item On arbitrary pairs in all Min and Max cones, the entrywise map
		induced by \(p_\alpha\) is Loewner convex, equivalently Loewner concave,
		if and only if
		$
		\alpha=1.
		$
		
		\item Under the comparable-pairs convention, the entrywise map induced by
		\(p_\alpha\) is Loewner convex on all Min and Max cones if and only if
		$
		\alpha=1
		\ \text{or}\
		\alpha\ge2.
		$
		
		\item Under the comparable-pairs convention, the entrywise map induced by
		\(p_\alpha\) is Loewner concave on all Min and Max cones if and only if
		$
		\alpha=1.
		$
	\end{enumerate}
\end{corollary}

	\begin{proof}
		For \(\alpha>0\), the function \(p_\alpha(t)=t^\alpha\) is nonnegative and
		nondecreasing on \([0,\infty)\). Hence the first assertion follows from
		Theorem~\ref{thm:entrywise-preserver}.
		
		On \((0,\infty)\), \(p_\alpha\) is positive and strictly increasing for
		every \(\alpha>0\). Therefore the second assertion follows from
		Proposition~\ref{prop:strict-preserver}.
		
For the third assertion, \(p_\alpha\) is continuous and nondecreasing on
\([0,\infty)\). By Theorem~\ref{thm:loewner-preserver}, it preserves
Loewner order on all Min and Max cones if and only if it is convex. The
function \(t^\alpha\) is convex on \([0,\infty)\) exactly when
\(\alpha\ge1\).

The fourth assertion follows from Proposition~\ref{prop:all-pairs-rigidity}, since
among the functions \(t^\alpha\) with \(\alpha>0\), only \(t\) is affine.

For Loewner convexity, if \(0<\alpha<1\), then \(p_\alpha\) is not convex.
If \(1<\alpha<2\), then
\[
p_\alpha'(t)=\alpha t^{\alpha-1}
\]
is not convex. If \(\alpha=1\), then \(p_\alpha\) is affine, while if
\(\alpha\ge2\), both \(p_\alpha\) and \(p_\alpha'\) are convex. Therefore,
by Theorem~\ref{thm:loewner-convexity-regularity}, the induced entrywise map is
Loewner convex if and only if
\[
\alpha=1
\qquad\text{or}\qquad
\alpha\ge2.
\]

Finally, consider Loewner concavity. If \(\alpha>1\), then \(p_\alpha\) is
strictly convex, so its Jensen gap is positive for \(d>0\) and
\(0<\lambda<1\), whereas Loewner concavity requires this gap to be
nonpositive. Thus Loewner concavity fails for \(\alpha>1\).

If \(0<\alpha<1\), fix \(d>0\) and \(0<\lambda<1\). For \(s>0\),
\[
\frac{\partial}{\partial s}J_{p_\alpha,\lambda}(s,d)
=
\alpha\left(
\lambda(s+d)^{\alpha-1}
+(1-\lambda)s^{\alpha-1}
-(s+\lambda d)^{\alpha-1}
\right)>0,
\]
because \(t^{\alpha-1}\) is strictly convex on \((0,\infty)\). Hence
\(J_{p_\alpha,\lambda}(s,d)\) is strictly increasing in \(s\), whereas
Proposition~\ref{prop:jensen-gap-characterization} requires it to be nonincreasing for
Loewner concavity. Thus Loewner concavity also fails for
\(0<\alpha<1\). When \(\alpha=1\), the Jensen gap is identically zero.
Consequently, the induced entrywise map is Loewner concave if and only if
\[
\alpha=1.
\]

	\end{proof}
	
	\begin{corollary}\label{cor:loewner-automorphism}
		Let \(f:[0,\infty)\to[0,\infty)\) be a bijection. Then the entrywise map
		\(A\mapsto f[A]\) is a Loewner order automorphism of all Min cones,
		equivalently of all Max cones, if and only if
		\[
		f(t)=ct,
		\qquad c>0.
		\]
	\end{corollary}
	
	\begin{proof}
		If \(f\) is an order automorphism, then both \(f\) and \(f^{-1}\) preserve
		Loewner order. By Theorem~\ref{thm:loewner-preserver}, both have
		increasing increments, and hence both are nondecreasing.
		
		Since \(f\) is a nondecreasing bijection from \([0,\infty)\) onto itself,
		it is continuous. Indeed, a jump discontinuity would leave a nonempty
		interval outside the range. By Theorem~\ref{thm:loewner-preserver},
		\(f\) is convex. The same argument applied to \(f^{-1}\) shows that
		\(f^{-1}\) is convex.
		
		Since \(f^{-1}\) is convex and increasing, \(f\) is concave; equivalently,
		the inverse of an increasing convex function is concave. Hence \(f\) is
		both convex and concave, so it is affine:
		\[
		f(t)=ct+d.
		\]
		Because \(f\) is a bijection from \([0,\infty)\) onto itself, \(f(0)=0\),
		so \(d=0\). Also \(c>0\). The converse is immediate.
	\end{proof}
	
	\section{Concluding remarks}
	
	The Min and Max cones provide a simple but useful setting in which several
	positivity notions coincide. For Min matrices,
	\[
	A_{\min}(x)\succeq0
	\quad\Longleftrightarrow\quad
	A_{\min}(x)\text{ is TN}
	\quad\Longleftrightarrow\quad
	0\le x_1\le\cdots\le x_n,
	\]
	and the Max analogue follows by reversing the sequence. This elementary cone
	geometry leads to entrywise preserver results that are much less restrictive
	than those for general positive semidefinite matrices.
	
	The main entrywise preserver theorem shows that positive semidefiniteness and
	total nonnegativity have the same preservers on Min and Max cones:
	\[
	f[A]\succeq0
	\quad\text{for all }A\in\mathcal C_{\min,n}\text{ and all }n
	\]
	if and only if
	\[
	f(t)\ge0
	\quad\text{and}\quad
	f\text{ is nondecreasing on }[0,\infty).
	\]
	The same condition characterizes total nonnegativity preservation and the Max
	analogue.
	
In the strict case, positive definiteness, nonsingular total
nonnegativity, the nonsingular inverse \(M\)-matrix property, and
oscillatory behavior fit into a common framework. For \(n\ge2\), the
inverse matrices have symmetric irreducible tridiagonal nonsingular
\(M\)-matrix structure. The entrywise preservers of these strict classes
are precisely the positive-valued strictly increasing functions.

Loewner order preservation is governed by increasing increments, and this
condition is equivalent to the entrywise function being nondecreasing and
convex without any prior regularity assumption. On the Min and Max cones,
Loewner convexity on arbitrary pairs is rigid and forces the entrywise
function to be affine. On comparable pairs it is
governed by the associated Jensen gaps. This condition itself forces
\(C^1\) regularity and is equivalent to convexity of both \(f\) and \(f'\).
The analogous concavity condition forces both functions to be concave. Thus,
for power functions,
\[
t^\alpha:
\qquad
	\begin{cases}
	\alpha>0,
	& \substack{\text{preserves PSD/TN and nonsingular inverse }M\text{-structure}\\
		\text{on strict cones}},\\
	\alpha\ge1,
	& \text{preserves Loewner order},\\
	\alpha=1,
	& \substack{\text{induces a Loewner convex or concave map}\\
		\text{on arbitrary pairs}},\\
	\alpha=1\text{ or }\alpha\ge2,
	& \substack{\text{induces a Loewner convex map}\\
		\text{on comparable pairs}},\\
	\alpha=1,
	& \substack{\text{induces a Loewner concave map}\\
		\text{on comparable pairs}}.
\end{cases}
\]

	Several directions remain open. For example, mixed Min--Max structures lead
	naturally to discrete Green matrices of the form
	\[
	g_{ij}=p_{\min(i,j)}q_{\max(i,j)}.
	\]
	These matrices contain Min and Max matrices as special cases. Their
total-nonnegativity, minor, Hadamard-product, and factorization structures
have been studied extensively. A natural continuation of the present work
is to characterize their entrywise and Loewner-order preservers.

\end{document}